\documentclass[12pt,leqno]{amsart}\usepackage[latin9]{inputenc}\usepackage{verbatim,amsthm,amstext,amssymb,color,amscd,bm}\makeatletter
\usepackage{hyperref}\hypersetup{colorlinks=true,citecolor=black,linkcolor=black}\usepackage[shortlabels]{enumitem}\numberwithin{equation}{section}
\theoremstyle{plain}\newtheorem{theorem}{Theorem}[section]\newtheorem{lemma}[theorem]{Lemma}
\theoremstyle{definition}\newtheorem{remark}[theorem]{Remark}\DeclareMathOperator{\mdim}{\mathrm{mdim}}
\DeclareMathOperator{\Widim}{\mathrm{Widim}}\DeclareMathOperator{\Sym}{\mathrm{Sym}}\DeclareMathOperator{\Map}{\mathrm{Map}}
\makeatother
\begin{document}
\title[Mean dimension of product spaces: a fundamental formula]{Mean dimension of product spaces:\\A fundamental formula}
\author[Lei Jin]{Lei Jin}
\address{Lei Jin: School of Mathematics, Sun Yat-sen University, Guangzhou, China}
\email{jinleim@mail.ustc.edu.cn}
\author[Yixiao Qiao]{Yixiao Qiao}
\address{Yixiao Qiao (Corresponding author): School of Mathematics and Statistics, Guangdong University of Technology, Guangzhou, China}
\email{yxqiao@mail.ustc.edu.cn}
\subjclass[2010]{37B99; 54F45.}\keywords{Mean dimension; Product space.}
\begin{abstract}
Mean dimension is a topological invariant of dynamical systems, which originates with Mikhail Gromov in 1999 and which was studied with deep applications around 2000 by Elon Lindenstrauss and Benjamin Weiss within the framework of amenable group actions. Let a countable discrete amenable group $G$ act continuously on compact metrizable spaces $X$ and $Y$. Consider the product action of $G$ on the product space $X\times Y$. The product inequality for mean dimension is well known: $\mdim(X\times Y,G)\le\mdim(X,G)+\mdim(Y,G)$, while it was unknown for a long time if the product inequality could be an equality. In 2019, Masaki Tsukamoto constructed the first example of two different continuous actions of $G$ on compact metrizable spaces $X$ and $Y$, respectively, such that the product inequality becomes strict. However, there is still one longstanding problem which remains open in this direction, asking if there exists a continuous action of $G$ on some compact metrizable space $X$ such that $\mdim(X\times X,G)<2\cdot\mdim(X,G)$. We solve this problem. Somewhat surprisingly, we prove, in contrast to (topological) dimension theory, a rather satisfactory theorem: If an infinite (countable discrete) amenable group $G$ acts continuously on a compact metrizable space $X$, then we have $\mdim(X^n,G)=n\cdot\mdim(X,G)$, for any positive integer $n$. Our product formula for mean dimension, together with the example and inequality (stated previously), eventually allows mean dimension of product actions to be fully understood.
\end{abstract}
\maketitle

\medskip

\section{Main result}
Mean dimension is a topological invariant of dynamical systems, which originates with Mikhail Gromov in 1999 and which was investigated with deep applications around 2000 by Elon Lindenstrauss and Benjamin Weiss within the framework of amenable group actions. The purpose of this paper is to establish a fundamental formula for \textit{mean dimension of product actions}. We shall state our main theorem very quickly in this section (Section 1). The definition of mean dimension and all the necessary terminologies can be found in Section 2. The proof of the main result is located in Section 3.

\medskip

Let us start with convention. Throughout this paper the symbol $\mathbb{N}$ will denote the set of positive integers. All acting groups are always assumed to be \textit{countable} and \textit{discrete}. If an amenable group $G$ acts continuously on a compact metrizable space $X$ then we denote its mean dimension by $\mdim(X,G)$, which takes values in $[0,+\infty]$.

\medskip

Let an amenable group $G$ act continuously on compact metrizable spaces $X_i$, respectively, where $i$ ranges over some subset $I$ of $\mathbb{N}$. Consider the product action of $G$ on the product space $\prod_{i\in I}X_i$. The product inequality for mean dimension (due to Lindenstrauss and Weiss \cite{LW}) is well known: $$\mdim(\prod_{i\in I}X_i,G)\le\sum_{i\in I}\mdim(X_i,G).$$ Nevertheless, it was unknown for a long time if the product inequality could always be an equality. In 2019, Masaki Tsukamoto \cite{Tsukamoto} successfully constructed the first example of two \textit{different} continuous actions of $G$ on compact metrizable spaces $X$ and $Y$, respectively, such that the product inequality becomes strict: $$\mdim(X\times Y,G)<\mdim(X,G)+\mdim(Y,G).$$ A serious reader may observe that in order to have a full understanding of mean dimension of product actions, there is still one longstanding issue that remains open, asking if it is possible for the two continuous actions $(X,G)$ and $(Y,G)$ (mentioned in the above example) to be \textit{essentially the same}. Formally, we study the problem in this direction as follows:
\begin{itemize}\item
For an (arbitrarily fixed) amenable group $G$, does there exist a continuous action of $G$ on some compact metrizable space $X$ such that $$\mdim(X\times X,G)<2\cdot\mdim(X,G)?$$
\end{itemize}
We solve this problem completely.

\medskip

First of all, let us make two observations on this issue here. On the one hand, we note that in dimension theory there is an example (to be precise, we refer to Lemma \ref{dimproduct}) of a compact metrizable space $K$ of (topological) dimension $\dim(K)$ so that the product space $K\times K$ satisfies $$\dim(K\times K)<2\cdot\dim(K).$$ On the other hand, we have made mention of Tsukamoto's example which is highly similar to such a classically known analogue that takes place in dimension theory.

Apparently, both of these two notable phenomena lead naturally to a seemingly plausible impression, i.e. it would be true that we could finally find a compact metrizable space $X$ (with a continuous action of $G$ on $X$) satisfying that $\mdim(X\times X,G)$ is strictly less than $2\cdot\mdim(X,G)$, as the (former) example in dimension theory would stimulate us to strengthen the (latter) construction of Tsukamoto's example in mean dimension theory with the help of some sufficiently refined method (which seems to be hopeful and which might be technically difficult).

However, this assertion turns out to be \textit{false}. Somewhat surprisingly, we prove a rather satisfactory theorem:

\medskip

\begin{theorem}[Main theorem]\label{main}
If an infinite amenable group $G$ acts continuously on a compact metrizable space $X$, then we have $$\mdim(X^n,G)=n\cdot\mdim(X,G)$$ for all positive integers $n$.
\end{theorem}

\medskip

\begin{remark}
Theorem \ref{main} also applies to $n\in\mathbb{N}\cup\{0\}\cup\{+\infty\}$ provided $0\cdot(+\infty)=(+\infty)\cdot0=0$. Indeed, this statement is obviously correct for $n=0$ if we set $X^0$ to be the one-point set. Moreover, with a slightly more effort assuming Theorem \ref{main} for any $n\in\mathbb{N}$ we are able to show that the statement is true for $n=+\infty$. In fact, there are two cases. If $\mdim(X,G)=0$, then $\mdim(X^\infty,G)=0$ follows directly from the product inequality for mean dimension. Now we suppose $\mdim(X,G)>0$. Since it is clear that $\mdim(X^\infty,G)\ge\mdim(X^n,G)=n\cdot\mdim(X,G)$ for every $n\in\mathbb{N}$ (by definition and by the statement of Theorem \ref{main} for all $n\in\mathbb{N}$), we have $\mdim(X^\infty,G)=+\infty$.
\end{remark}

\begin{remark}
If $G$ is a finite group (which is automatically amenable) then Theorem \ref{main} may be false. Notice that in this case we have by definition $\mdim(X,G)=\dim(X)/|G|$. As follows is an entire picture of the situation: If $X$ satisfies $\dim(X)=+\infty$, then $\mdim(X,G)=+\infty$. So does any of its self-product. Thus, the statement remains true for every $n\in\mathbb{N}\cup\{0,+\infty\}$ in this case. Now let us suppose that $X$ is finite dimensional. It follows from Lemma \ref{dimproduct} that for each $n\in\mathbb{N}$ \begin{align*}\mdim(X^n,G)=\begin{cases}n\cdot\mdim(X,G),&\text{if $\dim(X\times X)=2\dim(X)$}\\n\cdot\mdim(X,G)-(n-1)/|G|,&\text{otherwise}\end{cases}.\end{align*} Thus, in this case the statement fails if and only if $X$ does not satisfy $\dim(X\times X)=2\dim(X)$ and meanwhile $n$ does not belong to $\{0,1,+\infty\}$. In short, the \textit{exact range} to which the statement of Theorem \ref{main} does \textit{not} apply is where $G$ is a finite group, $X$ satisfies $\dim(X\times X)<2\dim(X)$, and $n\in\mathbb{N}\setminus\{1\}$.
\end{remark}

\medskip

In contrast to dimension theory, Theorem \ref{main} enables an unexpected behaviour in mean dimension theory to become clarified. Furthermore, our main theorem, together with Lindenstrauss--Weiss' inequality and Tsukamoto's example (stated previously), eventually allows mean dimension of product actions to be fully understood.

Our result is new even for $\mathbb{Z}$-actions. A novel point of the theorem is that the statement applies to the context of amenable group actions, whereas the proof goes through the framework of its sofic nature. The key ingredient of our idea is to produce \textit{different} sofic approximation sequences for the acting group, with respect to which, we consider the sofic mean dimension of a group action.

\medskip

\section{A brief review of mean dimension}
Both mean dimension and sofic groups originate with Misha Gromov around 1999. A systematic study of mean dimension in the context of amenable group actions was given around 2000 by Lindenstrauss and Weiss \cite{LW}. In 2013, Hanfeng Li \cite{Li} introduced the notion of sofic mean dimension which is a successful extension of the definition of mean dimension to the setting of sofic group actions, and further, Li built its connection with classical mean dimension. This section is devoted to all the precise notions and notations in relation to our result, and to collecting fundamental material on them.

\subsection{Sofic groups}
We denote by $|F|$ the cardinality of a set $F$. For every $d\in\mathbb{N}$ we write $[d]$ for the set $\{k\in\mathbb{N}:1\le k\le d\}$ and $\Sym(d)$ for the group of permutations of $[d]$. A group $G$ is \textbf{sofic} if there is a sequence $$\Sigma=\{\sigma_i:G\to\Sym(d_i)\}_{i\in\mathbb{N}}$$ together with a sequence $\{d_i\}_{i\in\mathbb{N}}\subset\mathbb{N}$ such that the following three conditions are satisfied:
\begin{align*}
&\bullet\quad\quad\lim_{i\to\infty}\frac{1}{d_i}|\{k\in[d_i]:\sigma_i(st)(k)=\sigma_i(s)\sigma_i(t)(k)\}|=1\quad\text{for all $s,t\in G$;}\\
&\bullet\quad\quad\lim_{i\to\infty}\frac{1}{d_i}|\{k\in[d_i]:\sigma_i(s)(k)\ne\sigma_i(t)(k)\}|=1\quad\text{for all distinct $s,t\in G$;}\\
&\bullet\quad\quad\lim_{i\to\infty}d_i=+\infty.
\end{align*}
Such a sequence $\Sigma$ is called a \textbf{sofic approximation sequence} for $G$.
\begin{remark}
Note that the third condition will be fulfilled automatically if $G$ is an infinite group.
\end{remark}
\begin{remark}
The sofic groups are a fairly extensive class, which contain in particular all amenable groups and all residually finite groups. However, it has not yet been confirmed if there exists a non-sofic group.
\end{remark}

\subsection{Product actions}
Let $G$ be a group. By the terminology ``$G$ \textbf{acts continuously on} a compact metrizable space $X$'' we understand a continuous mapping $$\Phi:G\times X\to X,\quad(g,x)\mapsto gx$$ satisfying $$\Phi(e,x)=x,\quad\Phi(gh,x)=\Phi(g,\Phi(h,x)),\quad\forall x\in X,\;\forall g,h\in G,$$ where $e$ is the identity element of the group $G$.

Let a group $G$ act continuously on compact metrizable spaces $X_n$, respectively, where $n$ ranges over some $R\in\{[r]:r\in\mathbb{N}\}\cup\{\mathbb{N}\}$. The \textbf{product action} of $G$ on the product space $\prod_{n\in R}X_n$ is defined as follows: $$g(x_n)_{n\in R}=(gx_n)_{n\in R},\quad\forall g\in G,\;\forall(x_n)_{n\in R}\in\prod_{n\in R}X_n.$$

\subsection{Dimension}
We denote by $\dim(K)$ the topological dimension (i.e. the Lebesgue covering dimension) of a compact metrizable space $K$. If the space $K$ is empty, then we set $\dim(K)=-\infty$. For a finite dimensional (nonempty) compact metrizable space $K$, since it was classically known that $$2\dim(K)-1\le\dim(K\times K)\le2\dim(K)$$ and since $\dim(K)$ must be a nonnegative integer, we have
\begin{itemize}\item either $\quad\dim(K\times K)=2\dim(K)$,\item or $\quad\dim(K\times K)=2\dim(K)-1$.\end{itemize}
For a friendly treatment of the following result in dimension theory we refer to \cite[Theorem 2.5]{Tsukamoto}.
\begin{lemma}\label{dimproduct}
Let $K$ be a finite dimensional compact metrizable space. Then for every $n\in\mathbb{N}$
$$\dim(K^n)=\begin{cases}n\dim(K),&\text{if $K$ satisfies $\dim(K\times K)=2\dim(K)$},\\n\dim(K)-n+1,&\text{otherwise}.\end{cases}$$
\end{lemma}
Let $X$ and $P$ be two compact metrizable spaces. Let $\rho$ be a compatible metric on $X$. For $\epsilon>0$ a continuous mapping $f:X\to P$ is called an \textbf{$\epsilon$-embedding} with respect to $\rho$ if $f(x)=f(x^\prime)$ implies $\rho(x,x^\prime)<\epsilon$, for all $x,x^\prime\in X$. Let $\Widim_\epsilon(X,\rho)$ be the minimum topological dimension $\dim(P)$ of a compact metrizable space $P$ which admits an $\epsilon$-embedding $f:X\to P$ with respect to $\rho$.
\begin{remark}
We may verify that the topological dimension of $X$ may be recovered by $\dim(X)=\lim_{\epsilon\to0}\Widim_\epsilon(X,\rho)$.
\end{remark}
Let $K$ be a compact metrizable space with a compatible metric $\rho$. For every $n\in\mathbb{N}$ we define on the product space $K^n$ two compatible metrics $\rho_2$ and $\rho_\infty$ as follows: $$\rho_2\left((x_i)_{i\in[n]},(y_i)_{i\in[n]}\right)=\sqrt{\frac1n\sum_{i\in[n]}(\rho(x_i,y_i))^2},$$ $$\rho_\infty\left((x_i)_{i\in[n]},(y_i)_{i\in[n]}\right)=\max_{i\in[n]}\rho(x_i,y_i).$$ We do not include $n\in\mathbb{N}$ in the notations $\rho_2$ and $\rho_\infty$ because it does not cause any ambiguity.

\subsection{Mean dimension}
A group $G$ is \textbf{amenable} if there exists a sequence $\{F_n\}_{n\in\mathbb{N}}$ of nonempty finite subsets of $G$ such that for any $g\in G$
$$\lim_{n\to\infty}\frac{|F_n\triangle gF_n|}{|F_n|}=0.$$
Such a sequence $\{F_n\}_{n\in\mathbb{N}}$ is called a \textbf{F\o lner sequence} of the group $G$.

Let an amenable group $G$ act continuously on a compact metrizable space $X$. Take a F\o lner sequence $\{F_n\}_{n\in\mathbb{N}}$ of $G$ and a compatible metric $\rho$ on $X$. For a nonempty finite subset $F$ of $G$ we set $$\rho_F(x,x^\prime)=\rho_\infty\left((gx)_{g\in F},(gx^\prime)_{g\in F}\right),\quad\forall\,x,x^\prime\in X.$$ It is clear that $\rho_F$ is also a compatible metric on $X$. The \textbf{mean dimension} of $(X,G)$ is defined by $$\mdim(X,G)=\lim_{\epsilon\to0}\lim_{n\to\infty}\frac{\Widim_\epsilon(X,\rho_{F_n})}{|F_n|}.$$ It is well known that the limits in the above definition always exist. The value $\mdim(X,G)$ is independent of the choices of a F\o lner sequence $\{F_n\}_{n\in\mathbb{N}}$ of $G$ and a compatible metric $\rho$ on $X$.

\subsection{Sofic mean dimension}
Suppose that $\Sigma=\{\sigma_i:G\to\Sym(d_i)\}_{i\in\mathbb{N}}$ is a sofic approximation sequence for a sofic group $G$ which acts continuously on a compact metrizable space $X$ equipped with a compatible metric $\rho$. For a finite subset $F$ of $G$, $\delta>0$ and a map $\sigma:G\to\Sym(d)$ (where $d\in\mathbb{N}$) we define $$\Map(\rho,F,\delta,\sigma)=\{\phi:[d]\to X:\rho_2(\phi\circ\sigma(s),s\phi)\le\delta,\,\forall s\in F\}.$$ We consider the set $\Map(\rho,F,\delta,\sigma)$ as a compact subspace of the product space $X^d$. The \textbf{sofic mean dimension} of $(X,G)$ with respect to $\Sigma$ is defined by $$\mdim_\Sigma(X,G)=\sup_{\epsilon>0}\inf_{F\subset G\text{ finite, }\,\delta>0}\limsup_{i\to\infty}\frac{\Widim_\epsilon\left(\Map(\rho,F,\delta,\sigma_i),\rho_\infty\right)}{d_i}.$$ The definition of $\mdim_\Sigma(X,G)$ does not depend on the compatible metrics $\rho$ on $X$. Nevertheless, it is not clear yet if there is an example of a sofic approximation sequence $\Sigma^\prime$ different from $\Sigma$, which leads to a different value $\mdim_{\Sigma^\prime}(X,G)$. We shall make use of the following theorem \cite[Section 3]{Li}.
\begin{lemma}\label{lithm}
If an infinite amenable group $G$ acts continuously on a compact metrizable space $X$ and if $\Sigma$ is a sofic approximation sequence for $G$, then $\mdim_\Sigma(X,G)=\mdim(X,G)$.
\end{lemma}

\medskip

\section{Proof of the main theorem}
Let $G$ be an infinite amenable group which acts continuously on a compact metrizable space $X$. We fix a positive integer $n$ in this section. Recall that $(X^n,G)$ denotes the product action of $G$ on the product space $X^n$. We shall prove $$\mdim(X^n,G)=n\cdot\mdim(X,G).$$

\medskip

Since the group $G$ is amenable, it is sofic. Therefore we may take a sofic approximation sequence for $G$: $$\Sigma=\{\sigma_i:G\to\Sym(d_i)\}_{i\in\mathbb{N}},$$ where $\{d_i\}_{i\in\mathbb{N}}$ is a sequence of positive integers with $d_i\to+\infty$ as $i\to+\infty$. We generate a new sofic approximation sequence for $G$ (confirmed below) as follows: $$\Sigma^{(n)}=\{\sigma^{(n)}_i:G\to\Sym(nd_i)\}_{i\in\mathbb{N}},$$ where for every $i\in\mathbb{N}$ the map $$\sigma^{(n)}_i:G\to\Sym(nd_i)$$ is defined by: $$\sigma^{(n)}_i(g)\left((j-1)n+l\right)=(\sigma_i(g)(j)-1)n+l,\,\quad\;\forall g\in G,\;\forall j\in[d_i],\;\forall l\in[n].$$

\medskip

\begin{lemma}
$\Sigma^{(n)}$ is a sofic approximation sequence for $G$.
\end{lemma}
\begin{proof}
Clearly, for every $i\in\mathbb{N}$ and $g\in G$ the map $\sigma^{(n)}_i(g):[nd_i]\to[nd_i]$ is a permutation of $[nd_i]$. Besides, it is straightforward to verify that for any $s,t\in G$, $j\in[d_i]$ and $l\in[n]$ we have $$\sigma^{(n)}_i(st)\left((j-1)n+l\right)=\sigma^{(n)}_i(s)\sigma^{(n)}_i(t)\left((j-1)n+l\right)\iff\sigma_i(st)(j)=\sigma_i(s)\sigma_i(t)(j),$$$$\sigma^{(n)}_i(s)\left((j-1)n+l\right)=\sigma^{(n)}_i(t)\left((j-1)n+l\right)\iff\sigma_i(s)(j)=\sigma_i(t)(j).$$ Since $\Sigma$ is a sofic approximation sequence for $G$, the assertion follows.
\end{proof}

\medskip

Let us consider the sofic mean dimension of $(X,G)$ and $(X^n,G)$ with respect to the sofic approximation sequences $\Sigma^{(n)}$ and $\Sigma$, respectively. These two values share the following relation.

\medskip

\begin{lemma}[Key lemma]\label{submain}
$$\mdim_{\Sigma^{(n)}}(X,G)=\frac1n\cdot\mdim_\Sigma(X^n,G).$$
\end{lemma}
\begin{proof}
We fix a compatible metric $\rho$ on $X$ in the proof. Let $\rho^{(n)}$ be the compatible metric $\rho_\infty$ on $X^n$. Let us consider two compact metric spaces as follows: $(X,\rho)$ and $(X^n,\rho^{(n)})$. We take $\epsilon>0$, $\delta>0$, a finite subset $F$ of $G$, and a positive integer $i$, arbitrarily and fix them temporarily.

We note that both of the following two sets: $$\Map(\rho,F,\delta,\sigma^{(n)}_i)=\{\phi:[nd_i]\to X:\rho_2(\phi\circ\sigma^{(n)}_i(s),s\phi)\le\delta,\,\forall s\in F\}$$$$\Map(\rho^{(n)},F,\delta,\sigma_i)=\{\phi:[d_i]\to X^n:\rho^{(n)}_2(\phi\circ\sigma_i(s),s\phi)\le\delta,\,\forall s\in F\}$$ can be regarded as compact subspaces of the product space $X^{nd_i}=(X^n)^{d_i}$. More explicitly, the point here is that we identify $X^{nd_i}$ with $(X^n)^{d_i}$. We notice that the construction of the sofic approximation sequence $\Sigma^{(n)}$ for $G$ and the definition of the product action $(X^n,G)$ ensure that the terms $\phi\circ\sigma^{(n)}_i(s)$ and $\phi\circ\sigma_i(s)$ agree, i.e. $$\phi\circ\sigma^{(n)}_i(s)=\phi\circ\sigma_i(s),\quad\;\forall\phi\in X^{nd_i}=(X^n)^{d_i},\;\forall s\in F.$$ Further, we also remark that $\rho_\infty$ defined on $X^{nd_i}$ corresponds to $\rho^{(n)}_\infty$ defined on $(X^n)^{d_i}$, namely $$\rho_\infty(\psi,\psi^\prime)=\rho^{(n)}_\infty(\psi,\psi^\prime),\quad\,\forall\psi,\psi^\prime\in X^{nd_i}=(X^n)^{d_i},$$ while $\rho_2$ defined on $X^{nd_i}$ and $\rho^{(n)}_2$ defined on $(X^n)^{d_i}$ satisfy the inequality: $$\frac{1}{\sqrt{n}}\cdot\rho^{(n)}_2(\psi,\psi^\prime)\le\rho_2(\psi,\psi^\prime)\le\rho^{(n)}_2(\psi,\psi^\prime),\quad\,\forall\psi,\psi^\prime\in X^{nd_i}=(X^n)^{d_i}.$$ The above observation implies that $$\Map(\rho^{(n)},F,\delta,\sigma_i)\subset\Map(\rho,F,\delta,\sigma^{(n)}_i)\subset\Map(\rho^{(n)},F,\sqrt{n}\delta,\sigma_i).$$ It follows that
\begin{align*}&\Widim_\epsilon\left(\Map(\rho^{(n)},F,\delta,\sigma_i),\rho^{(n)}_\infty\right)\\\le&\Widim_\epsilon\left(\Map(\rho,F,\delta,\sigma^{(n)}_i),\rho_\infty\right)\\\le&\Widim_\epsilon\left(\Map(\rho^{(n)},F,\sqrt{n}\delta,\sigma_i),\rho^{(n)}_\infty\right).\end{align*}

Since $\epsilon>0$, $\delta>0$, a finite subset $F\subset G$ and $i\in\mathbb{N}$ (which we took in the beginning of the proof) are arbitrary, we deduce that $$\mdim_{\Sigma^{(n)}}(X,G)=\frac1n\cdot\mdim_\Sigma(X^n,G).$$ Thus, we end the proof.
\end{proof}
\begin{remark}
The equality established in Lemma \ref{submain} is generally true for all sofic group actions $(X,G)$ and all positive integers $n$. The acting group $G$ in this lemma is not required to be infinite.
\end{remark}

\medskip

We are now able to prove Theorem \ref{main}. The key lemma (Lemma \ref{submain}) indicates that $$\mdim_{\Sigma^{(n)}}(X,G)=\frac1n\cdot\mdim_\Sigma(X^n,G).$$ By Lemma \ref{lithm}, sofic mean dimension (with respect to any sofic approximation sequence) will coincide with (classical) mean dimension, as the acting group $G$ is infinite. Thus, we conclude with $$\mdim(X,G)=\frac1n\cdot\mdim(X^n,G).$$

\medskip

\begin{remark}
We explain shortly about the difficulty with this problem. Let $G$ be an infinite amenable group which acts continuously on a compact metrizable space $X$. We fix a positive integer $n$ and a F\o lner sequence $\{F_k\}_{k\in\mathbb{N}}$ of $G$. We recall that $$\mdim(X^n,G)=\lim_{\epsilon\to0}\lim_{n\to\infty}\frac{\Widim_\epsilon(X^n,(\rho_\infty)_{F_k})}{|F_k|},$$$$n\cdot\mdim(X,G)=\lim_{\epsilon\to0}\lim_{n\to\infty}\frac{n\cdot\Widim_\epsilon(X,\rho_{F_k})}{|F_k|},$$ where $\rho$ and $\rho_\infty$ are compatible metrics on $X$ and $X^n$, respectively. To show $$\mdim(X^n,G)\ge n\cdot\mdim(X,G),$$ a main issue is how to estimate the term $n\cdot\Widim_\epsilon(X,\rho_{F_k})$ from above with terms such as some variants of $\Widim_\epsilon(X^n,(\rho_\infty)_{F_k})$. We overcome this obstacle. The strategy we adopted is to consider \textit{different} approximation sequences in the limits. For a systematic treatment we went through an approach of sofic mean dimension. To make it clearer, let us focus on the case of $\mathbb{Z}$-actions. More precisely, let $(X,d)$ be a compact metric space and $T:X\to X$ a homeomorphism. For convenience we change our notations here, which apply only to the remark. For every positive integer $k$ we write $d_k$ for the compatible metric on $X$ defined by $$d_k(x,x^\prime)=\max_{0\le i<k}d(T^ix,T^ix^\prime),\quad\forall\,x,x^\prime\in X.$$ For simplicity we assume $n=2$. The product action of $\mathbb{Z}$ on the product space $X\times X$ is denoted by $T\times T$. We take a compatible metric $d\times d$ on $X\times X$ as follows: $$d\times d\,((x_1,x_2),(x^\prime_1,x^\prime_2))=\max\{d(x_1,x^\prime_1),d(x_2,x^\prime_2)\},\,\quad\forall\,(x_1,x_2),\,(x^\prime_1,x^\prime_2)\;\in X\times X.$$ To estimate $2\cdot\mdim(X,T)$ from above, we have to turn to an alternative expression (replacing $N$ by $2N$ in the midst of the equality below): $$2\cdot\mdim(X,T)=\lim_{\epsilon\to0}\lim_{N\to+\infty}\frac{2\cdot\Widim_\epsilon(X,d_N)}{N}=\lim_{\epsilon\to0}\lim_{N\to+\infty}\frac{\Widim_\epsilon(X,d_{2N})}{N}.$$ Therefore, in order to show $$2\cdot\mdim(X,T)\le\mdim(X\times X,T\times T)=\lim_{\epsilon\to0}\lim_{N\to+\infty}\frac{\Widim_\epsilon(X\times X,(d\times d)_N)}{N}$$ it suffices to prove $$\Widim_\epsilon(X,d_{2N})\le\Widim_\epsilon(X\times X,(d\times d)_N)$$ for $\epsilon>0$ and $N\in\mathbb{N}$. This will be deduced from the following statement: The continuous mapping $$X\to X\times X,\quad x\mapsto(x,T^Nx)$$ is distance-increasing (actually it is distance-preserving) with respect to $(X,d_{2N})$ and $(X\times X,(d\times d)_N)$, i.e. $$d_{2N}(x,x^\prime)=\max_{0\le i\le2N-1}d(T^ix,T^ix^\prime)=(d\times d)_N\left((x,T^Nx),(x^\prime,T^Nx^\prime)\right),\quad\forall\,x,x^\prime\in X.$$
\end{remark}

\medskip


\bigskip

\begin{thebibliography}{9999999}

\bibitem[Li13]{Li}
Hanfeng Li.
Sofic mean dimension.
Advances in Mathematics 244 (2013), 570--604.

\bibitem[LW00]{LW}
Elon Lindenstrauss, Benjamin Weiss.
Mean topological dimension.
Israel Journal of Mathematics 115 (2000), 1--24.

\bibitem[Tsu19]{Tsukamoto}
Masaki Tsukamoto.
Mean dimension of full shifts.
Israel Journal of Mathematics 230 (2019), 183--193.

\end{thebibliography}
\end{document}